\documentclass[12pt]{article}
\usepackage{amssymb}
\usepackage{amsthm}
\usepackage{amsmath}
\usepackage{graphicx}
\usepackage{enumerate}

\DeclareMathOperator{\Max}{Max}
\DeclareMathOperator{\Min}{Min}

\newtheorem{theorem}{Theorem}[section]
\newtheorem{definition}[theorem]{Definition}
\newtheorem{lemma}[theorem]{Lemma}
\newtheorem{proposition}[theorem]{Proposition}
\newtheorem{remark}[theorem]{Remark}
\newtheorem{example}[theorem]{Example}

\title{Implication in finite posets with pseudocomplemented sections}
\author{Ivan~Chajda and Helmut~L\"anger}
\date{}
\begin{document}

\footnotetext{Support of the research by the Austrian Science Fund (FWF), project I~4579-N, and the Czech Science Foundation (GA\v CR), project 20-09869L, entitled ``The many facets of orthomodularity'', as well as by \"OAD, project CZ~02/2019, entitled ``Function algebras and ordered structures related to logic and data fusion'', and, concerning the first author, by IGA, project P\v rF~2021~030, is gratefully acknowledged.}

\maketitle

\begin{abstract}
It is well-known that relatively pseudocomplemented lattices can serve as an algebraic semantics of intuitionistic logic. To extend the concept of relative pseudocomplementation to non-distributive lattices, the first author introduced so-called sectionally pseudocomplemented lattices, i.e.\ lattices with top element $1$ where for every element $y$ the interval $[y,1]$, the so called  section, is pseudocomplemented. We extend this concept to posets with top element. Our goal is to show that such a poset can be considered as an algebraic semantics for a certain kind of more general intuitionistic logic provided an implication is introduced as shown in the paper. We prove some properties of such an implication. This implication is ``unsharp'' in the sense that the value for given entries need not be a unique element, but may be a subset of the poset in question. On the other hand, all of these values are as high as possible. We show that this implication even determines the poset and if a new operator $\odot$ is introduced in an ``unsharp'' way, such structure forms an ``unsharply'' residuated poset.
\end{abstract}



\section{Introduction}

Relatively pseudocomplemented lattices, often called Heyting algebras (see e.g.\ \cite I and \cite{K80}) or Brouwerian lattices (see e.g.\ \cite{K81}), arise from intuitionistic logic and were first investigated by T.~Skolem about 1920, see also \cite F and \cite{Ba}. For a detailed development see e.g.\ \cite{Cu}. Within this context, the relative pseudocomplement $x*y$ of $x$ with respect to $y$ is usually considered as intuitionistic implication, see e.g.\ \cite N or \cite{Cu}.

Hence, in relatively pseudocomplemented lattices we define
\[
x\rightarrow y:= x*y.
\]
It is well-known that every finite pseudocomplemented lattice is distributive. To extend investigations in intuitionistic logic also to the non-distributive case, the first author introduced so-called {\em sectionally pseudocomplemented lattices}, see \cite{Ch} and \cite{CR}. These are lattices with a top element where for every element $y$ and every element $x$ in the interval (so-called {\em section}) $[y,1]$ there exists a pseudocomplement $x^y$ of $x$ with respect to $y$. Putting
\begin{equation}\label{equ1}
x\rightarrow y:=(x\vee y)^y
\end{equation}
the situation becomes formally analogous to the case of relatively pseudocomplemented lattices. For the typical case, consider the lattice depicted in Figure~1:

\vspace*{-2mm}

\begin{center}
\setlength{\unitlength}{7mm}
\begin{picture}(6,8)
\put(3,1){\circle*{.3}}
\put(5,3){\circle*{.3}}
\put(1,4){\circle*{.3}}
\put(5,5){\circle*{.3}}
\put(3,7){\circle*{.3}}
\put(3,1){\line(-2,3)2}
\put(3,1){\line(1,1)2}
\put(3,7){\line(-2,-3)2}
\put(3,7){\line(1,-1)2}
\put(5,3){\line(0,1)2}
\put(2.85,.25){$0$}
\put(5.4,2.85){$a$}
\put(.3,3.85){$b$}
\put(5.4,4.85){$c$}
\put(2.85,7.4){$1$}
\put(2.2,-.75){{\rm Fig.~1}}
\end{picture}
\end{center}

\vspace*{4mm}

It is evident that this lattice has pseudocomplemented sections, but the lattice is neither relatively pseudocomplemented (since the relative pseudocomplement of $c$ with respect to $a$ does not exist) nor distributive.

The operation tables for $x^y$ and $\rightarrow$ look as follows:
\[
\begin{array}{c|ccccc}
x^y & 0 & a & b & c & 1 \\
\hline
 0  & 1 & - & - & - & - \\
 a  & b & 1 & - & - & - \\
 b  & c & - & 1 & - & - \\
 c  & b & a & - & 1 & - \\
 1  & 0 & a & b & c & 1
\end{array}
\quad\quad\quad
\begin{array}{c|ccccc}
\rightarrow & 0 & a & b & c & 1 \\
\hline
      0     & 1 & 1 & 1 & 1 & 1 \\
			a     & c & 1 & b & 1 & 1 \\
			b     & c & a & 1 & c & 1 \\
			c     & b & a & b & 1 & 1 \\
			1     & 0 & a & b & c & 1.
\end{array}
\]
The notion of relatively pseudocomplemented lattices was extended to posets, see e.g.\ \cite{CLP}. It is useful when a reduct of intuitionistic logic is considered where one studies only the connective implication but not other connectives like disjunction or conjunction. Let us note that in intuitionistic logic, the connectives implication, conjunction and disjunction are independent.

\section{Posets with pseudocomplemented sections}

To extend our study also to (not necessarily relatively pseudocomplemented) posets with pseudocomplemented sections, let us introduce several necessary concepts.

Let $(P,\leq)$ be a poset $a,b\in P$ and $A,B\subseteq P$. We say $A<B$ if $x\leq y$ for all $x\in A$ and $y\in B$. Instead of $\{a\}<\{b\}$, $\{a\}<B$ and $A<\{b\}$ we simply write $a<b$, $a<B$ and $A<b$, respectively. Analogously we proceed with the relational symbols $\leq$, $>$ and $\geq$. Denote by
\[
L(A):=\{x\in P\mid x\leq A\}\text{ and }U(A):=\{x\in P\mid A\leq x\}
\]
the so-called {\em lower} and {\em upper cone} of $A$, respectively. Instead of $L(\{a\})$, $L(\{a,b\})$, $L(A\cup\{a\})$, $L(A\cup B)$ and $L\big(U(A)\big)$ we simply write $L(a)$, $L(a,b)$, $L(A,a)$, $L(A,B)$ and $LU(A)$, respectively. Analogously, we proceed in similar cases. Denote the set of all minimal and maximal elements of $A$ by $\Min A$ and $\Max A$, respectively.

Recall that a {\em pseudocomplemented poset} is an ordered quadruple $(P,\leq,{}^*,0)$ where $(P,\leq,0)$ is a poset with bottom element $0$ and $^*$ is a unary operation on $P$ such that for all $x\in P$, $x^*$ is the greatest element of $(P,\leq)$ satisfying $L(x,x^*)=0$. (Here and in the following, we often identify singletons with their unique element.) This means that $x\wedge x^*$ exists for each $x\in P$ and $x\wedge x^*=0$.

Let us mention that in every logic, both classical or non-classical, a prominent role plays the logical connective implication. The reason is that implication enables logical deduction, i.e.\ the derivation of new propositions from given ones. In order to study a logic based on a poset, one cannot expect that the result of implication will be uniquely determined. This means that the result of the implication $x\rightarrow y$ for given elements $x$ and $y$ of a given poset $P$ would be a subset of $P$, not necessarily a singleton. This is the reason why we will call such an implication ``unsharp''. On the other hand, we ask such an unsharp implication to satisfy the rules and properties usually satisfied by an implication and, moreover, the results of our implication should be as high as possible. We introduce such an unsharp implication within the next section. In Proposition~\ref{prop1} we show that our implication satisfies properties similar to those satisfied by the standard implication. We also show that the values of results of our implication are usually higher than those for implication of intuitionistic logic based on relative pseudocomplementation. In the last section we introduce also an unsharp connective conjunction which is connected with our implication via a certain kind of adjointness.

\begin{definition}
A {\em finite poset with pseudocomplemented sections} is an ordered quadruple $\big(P,\leq,(^y;y\in P),1\big)$ where $(P,\leq,1)$ is a finite poset with top element $1$ and for every $y\in P$, $([y,1],\leq,{}^y,y)$ is a pseudocomplemented poset. For every $y\in P$ and every subset $B$ of $[y,1]$ put $B^y:=\{b^y\mid b\in B\}$. Finally, for all $x,y\in P$ define the implication $x\rightarrow y$ as follows:
\[
x\rightarrow y:=\big(\Min U(x,y)\big)^y.
\]
A {\em finite poset with $0$ and pseudocomplemented sections} is an ordered quintuple $\big(P,\leq,(^y;y\in P),0,1\big)$ where $\big(P,\leq,(^y;y\in P),1\big)$ is a finite poset with pseudocomplemented sections and $0$ is the bottom element of $(P,\leq)$.
\end{definition}

Observe that because of $1\in U(x,y)$ we have $\Min U(x,y)\neq\emptyset$.

\begin{remark}\label{rem1}
If $\big(P,\leq,(^y;y\in P),1\big)$ is a finite poset with pseudocomplemented sections, $b\in P$ and $a\in[b,1]$ then
\[
a^b=\max\{x\in P\mid L(a,x)\cap[b,1]=b\}.
\]
\end{remark}

Hence, in general, $\rightarrow$ is not a binary operation on $P$ but an operator assigning to each element of $P^2$ a non-empty subset of $P$. The almost obvious relationship between the sectional pseudocomplementation and the operator $\rightarrow$ is as follows. 

\begin{lemma}\label{lem1}
Let $\big(P,\leq,(^y;y\in P),1\big)$ be a finite poset with pseudocomplemented sections and $a,b\in P$. Then the following hold:
\begin{enumerate}[{\rm(i)}]
\item If $a\vee b$ exists in $(P,\leq)$ then $a\rightarrow b=(a\vee b)^b$,
\item if $b\leq a$ then $a\rightarrow b=a^b$.
\end{enumerate}
\end{lemma}

\begin{proof}
\
\begin{enumerate}[(i)]
\item If $a\vee b$ exists in $(P,\leq)$ then
\[
a\rightarrow b=\big(\Min U(a,b)\big)^b=\big(\Min U(a\vee b)\big)^b=(a\vee b)^b.
\]
\item if $b\leq a$ then because of (i) we have $a\rightarrow b=(a\vee b)^b=a^b$.
\end{enumerate}
\end{proof}

In what follows we list some elementary but important properties of this implication. We can see that these are analogous to know properties of implication in classical and non-classical propositional calculus.

\begin{proposition}\label{prop1}
Let $\big(P,\leq,(^y;y\in P),1\big)$ be a finite poset with pseudocomplemented sections and $a,b\in P$. Then the following hold:
\begin{enumerate}[{\rm(i)}]
\item $a\leq b$ if and only if $a\rightarrow b=1$,
\item if $a\vee b$ exists then $(a\vee b)\rightarrow b=a\rightarrow b$,
\item $1\rightarrow a=a$,
\item $a\leq b\rightarrow a$,
\item $a\rightarrow(b\rightarrow a)=1$.
\end{enumerate}
\end{proposition}

\begin{proof}
\
\begin{enumerate}[(i)]
\item The following are equivalent:
\begin{align*}
                      a & \leq b, \\
            \Min U(a,b) & =b, \\
                      x & =b\text{ for all }x\in\Min U(a,b), \\
                    x^b & =1\text{ for all }x\in\Min U(a,b), \\
\big(\Min U(a,b)\big)^b & =1, \\
         a\rightarrow b & =1,
\end{align*}
\item if $a\vee b$ exists then
\[
(a\vee b)\rightarrow b=\big(\Min U(a\vee b,b)\big)^b=\big(\Min U(a\vee b)\big)^b=\big(\Min U(a,b)\big)^b=a\rightarrow b,
\]
\item
\[
1\rightarrow a=\big(\Min U(1,a)\big)^a=\big(\Min U(1)\big)^a=1^a=a,
\]
\item $a\leq\big(\Min U(b,a)\big)^a=b\rightarrow a$,
\item this follows from (iii) and from (i) of Lemma~\ref{lem1}.
\end{enumerate}
\end{proof}

The next result shows that under appropriate assumptions our unsharp implication satisfies important properties already known from standard implication.

\begin{proposition}
Let $\big(P,\leq,(^y;y\in P),1\big)$ be a finite poset with pseudocomplemented sections and $a,b,c\in P$. Then the following hold:
\begin{enumerate}[{\rm(i)}]
\item If $a\leq b$ and $a\vee c$ exists in $(P,\leq)$ then $b\rightarrow c\leq a\rightarrow c$,
\item if $a\vee b$ exists in $(P,\leq)$ then $a\leq(a\rightarrow b)\rightarrow b$,
\item if $a\vee b$ exists in $(P,\leq)$ then $a\rightarrow b=\big((a\rightarrow b)\rightarrow b\big)\rightarrow b$.
\end{enumerate}
\end{proposition}

\begin{proof}
\
\begin{enumerate}[(i)]
\item Since $a\vee c$ exists in $(P,\leq)$, we have $a\rightarrow c=(a\vee c)^c$ according to (i) of Lemma~\ref{lem1}. Now, by (P2), everyone of the following assertions implies the next one:
\begin{align*}
             a & \leq b, \\
        U(b,c) & \subseteq U(a,c), \\
   \Min U(b,c) & \subseteq U(a\vee c), \\
       a\vee c & \leq x\text{ for all }x\in\Min U(b,c), \\
           x^c & \leq(a\vee c)^c=a\rightarrow c\text{ for all }x\in\Min U(b,c), \\
b\rightarrow c & =\big(\Min U(b,c)\big)^c\leq a\rightarrow c
\end{align*}
\item Because of (P3) and (i) and (ii) of Lemma~\ref{lem1} we have
\[
a\leq a\vee b\leq\big((a\vee b)^b\big)^b=(a\rightarrow b)\rightarrow b.
\]
\item Because of (i) of Lemma~\ref{lem1}, (P4) and (ii) of Lemma~\ref{lem1} we have
\[
a\rightarrow b=(a\vee b)^b=\Big(\big((a\vee b)^b\big)^b\Big)^b=\big((a\rightarrow b)\rightarrow b\big)\rightarrow b.
\]
\end{enumerate}
\end{proof}

Let $\mathbf P=(P,\leq)$ be a poset and $a,b\in P$. Recall the following definitions.
\begin{itemize}
\item The greatest element $x$ of $P$ satisfying $L(a,x)\subseteq L(b)$ is called the {\em relative pseudocomplement} $a*b$ of $a$ with respect to $b$. The poset $\mathbf P$ is called {\em relatively pseudocomplemented} if any two elements of $P$ have a relative pseudocomplement, see \cite{CLP} and \cite V.
\item The greatest element $x$ of $P$ satisfying $L(U(a,b),x)=L(b)$ is called the {\em sectional pseudocomplement} $a\circ b$ of $a$ with respect to $b$. The poset $\mathbf P$ is called {\em sectionally pseudocomplemented} if any two elements of $P$ have a sectional pseudocomplement.
\end{itemize}

\begin{remark}
Let $(P,\leq,1)$ be a poset with top element $1$ and $a,b\in P$ with $b\leq a$. Further assume that the sectional pseudocomplement $a\circ b$ of $a$ with respect to $b$ and the pseudocomplement of $a^b$ of $a$ in $[b,1]$ exist. Then $a\circ b\leq a^b$.
\end{remark}

\begin{proof}
Since $b\in L(b)=L\big(U(a,b),a\circ b\big)$, we have $b\leq a\circ b$. Moreover,
\[
L(a,a\circ b)\cap[b,1]=L\big(U(a),a\circ b\big)\cap[b,1]=L\big(U(a,b),a\circ b\big)\cap[b,1]=L(b)\cap[b,1]=\{b\}.
\]
Hence $a\circ b\leq a^b$.
\end{proof}

Let us note that the sectional pseudocomplement is not the same as the pseudocomplement in the corresponding section. For example, consider the poset depicted in Fig.~2. Then $a\notin[b,1]$. Thus the pseudocomplement of $a$ in the section $[b,1]$, i.e.\ $a^b$, does not exist. On the other hand, the sectional pseudocomplement $a\circ b$ of $a$ with respect to $b$ exists and is equal to $b$ because $b$ is the greatest element $x$ satisfying $L\big(U(a,b),x\big)=L(b)$ since $U(a,b)=\{c,d,1\}$. It is worth noticing that $a\circ b$ differs from our unsharp implication $a\rightarrow b$ because $a\rightarrow b=\{c,d\}$.

\begin{example}\label{ex2}
The poset shown in Figure~2:

\vspace*{-2mm}

\begin{center}
\setlength{\unitlength}{7mm}
\begin{picture}(6,10)
\put(3,1){\circle*{.3}}
\put(1,3){\circle*{.3}}
\put(5,3){\circle*{.3}}
\put(1,7){\circle*{.3}}
\put(5,7){\circle*{.3}}
\put(3,9){\circle*{.3}}
\put(1,3){\line(1,-1)2}
\put(1,3){\line(1,1)4}
\put(1,3){\line(0,1)4}
\put(5,3){\line(-1,-1)2}
\put(5,3){\line(-1,1)4}
\put(5,3){\line(0,1)4}
\put(3,9){\line(-1,-1)2}
\put(3,9){\line(1,-1)2}
\put(2.85,.25){$0$}
\put(.3,2.85){$a$}
\put(5.4,2.85){$b$}
\put(.3,6.85){$c$}
\put(5.4,6.85){$d$}
\put(2.85,9.4){$1$}
\put(2.2,-.75){{\rm Fig.~2}}
\end{picture}
\end{center}

\vspace*{4mm}

has pseudocomplemented sections and is simultaneously relatively pseudocomplemented. The tables for $x^y$, $\rightarrow$ and $*$ look as follows:
\[
\begin{array}{c|cccccc}
x^y & 0 & a & b & c & d & 1 \\
\hline
 0  & 1 & - & - & - & - & - \\
 a  & b & 1 & - & - & - & - \\
 b  & a & - & 1 & - & - & - \\
 c  & 0 & d & d & 1 & - & - \\
 d  & 0 & c & c & - & 1 & - \\
 1  & 0 & a & b & c & d & 1
\end{array}
\quad
\begin{array}{c|cccccc}
\rightarrow & 0 &    a    &    b    & c & d & 1 \\
\hline
      0     & 1 &    1    &    1    & 1 & 1 & 1 \\
			a     & b &    1    & \{c,d\} & 1 & 1 & 1 \\
			b     & a & \{c,d\} &    1    & 1 & 1 & 1 \\
			c     & 0 &    d    &    d    & 1 & d & 1 \\
			d     & 0 &    c    &    c    & c & 1 & 1 \\
			1     & 0 &    a    &    b    & c & d & 1
\end{array}
\quad
\begin{array}{c|cccccc}
* & 0 & a & b & c & d & 1 \\
\hline
0 & 1 & 1 & 1 & 1 & 1 & 1 \\
a & b & 1 & b & 1 & 1 & 1 \\
b & a & a & 1 & 1 & 1 & 1 \\
c & 0 & a & b & 1 & d & 1 \\
d & 0 & a & b & c & 1 & 1 \\
1 & 0 & a & b & c & d & 1.
\end{array}
\]
The intuitionistic implication, i.e.\ the relative pseudocomplement $*$ differs from our ``unsharp'' implication $\rightarrow$, e.g.\ $a*b=b$ whereas $a\rightarrow b=\{c,d\}$. Hence, although $a\rightarrow b$ is an ``unsharp'' implication because its result is a two-element subset of $P$, its values $c$ and $d$ are greater than the value of intuitionistic implication $a*b$.
\end{example}

\begin{example}\label{ex1}
The poset shown in Figure~3:

\vspace*{-2mm}

\begin{center}
\setlength{\unitlength}{7mm}
\begin{picture}(6,10)
\put(3,1){\circle*{.3}}
\put(2,2){\circle*{.3}}
\put(1,3){\circle*{.3}}
\put(5,3){\circle*{.3}}
\put(1,7){\circle*{.3}}
\put(5,7){\circle*{.3}}
\put(3,9){\circle*{.3}}
\put(1,3){\line(1,-1)2}
\put(1,3){\line(1,1)4}
\put(1,3){\line(0,1)4}
\put(5,3){\line(-1,-1)2}
\put(5,3){\line(-1,1)4}
\put(5,3){\line(0,1)4}
\put(3,9){\line(-1,-1)2}
\put(3,9){\line(1,-1)2}
\put(2.85,.25){$0$}
\put(1.3,1.85){$a$}
\put(.3,2.85){$b$}
\put(5.4,2.85){$c$}
\put(.3,6.85){$d$}
\put(5.4,6.85){$e$}
\put(2.85,9.4){$1$}
\put(2.2,-.75){{\rm Fig.~3}}
\end{picture}
\end{center}

\vspace*{5mm}

has pseudocomplemented sections, but is not relatively pseudocomplemented since the relative pseudocomplement of $b$ with respect to $a$ does not exist. The tables for $x^y$ and $\rightarrow$ look as follows:
\[
\begin{array}{c|ccccccc}
x^y & 0 & a & b & c & d & e & 1 \\
\hline
 0  & 1 & - & - & - & - & - & - \\
 a  & c & 1 & - & - & - & - & - \\
 b  & c & a & 1 & - & - & - & - \\
 c  & b & - & - & 1 & - & - & - \\
 d  & 0 & a & e & e & 1 & - & - \\
 e  & 0 & a & d & d & - & 1 & - \\
 1  & 0 & a & b & c & d & e & 1.
\end{array}
\quad\quad\quad
\begin{array}{c|ccccccc}
\rightarrow & 0 & a &    b    &    c    & d & e & 1 \\
\hline
      0     & 1 & 1 &    1    &    1    & 1 & 1 & 1 \\
			a     & c & 1 &    1    & \{d,e\} & 1 & 1 & 1 \\
			b     & c & a &    1    & \{d,e\} & 1 & 1 & 1 \\
			c     & b & a & \{d,e\} &    1    & 1 & 1 & 1 \\
			d     & 0 & a &    e    &    e    & 1 & e & 1 \\
			e     & 0 & a &    d    &    d    & d & 1 & 1 \\
			1     & 0 & a &    b    &    c    & d & e & 1.
\end{array}
\]
\end{example}

It is a question if, having an operator $\rightarrow$ on a finite set $A$, it can be converted into a poset with pseudocomplemented sections. For this, we introduce the following structure.

\section{Implication algebras}

Our next goal is to show that this unsharp implication in fact determines the given finite poset with pseudocomplemented sections. For this purpose we define the following concept.

\begin{definition}\label{def1}
A {\em finite {\rm I}-algebra} is an ordered triple $(A,\rightarrow,1)$ with a finite set $A$, an operator $\rightarrow:A^2\rightarrow2^A\setminus\{\emptyset\}$ and $1\in A$ satisfying the following conditions:
\begin{enumerate}
\item[{\rm(I1)}] $x\rightarrow x\approx x\rightarrow1\approx1$,
\item[{\rm(I2)}] $x\rightarrow y=y\rightarrow x=1\Rightarrow x=y$,
\item[{\rm(I3)}] $x\rightarrow y=y\rightarrow z=1\Rightarrow x\rightarrow z=1$,
\item[{\rm(I4)}] $y\rightarrow z=z\rightarrow x=z\rightarrow(x\rightarrow y)=1\Rightarrow z=y$,
\item[{\rm(I5)}] $\big(y\rightarrow x=y\rightarrow u=1\text{ and }(y\rightarrow z=z\rightarrow x=z\rightarrow u=1\Rightarrow z=y)\big)\Rightarrow u\rightarrow(x\rightarrow y)=1$,
\item[{\rm(I6)}] $x\rightarrow y=\{z\rightarrow y\mid x\rightarrow z=y\rightarrow z=1,\text{ and }x\rightarrow u=y\rightarrow u=u\rightarrow z=1\Rightarrow u=z\}$.
\end{enumerate}
\end{definition}

Now we can state and prove the following result.
  
\begin{theorem}\label{th1}
Let $\mathbf P=\big(P,\leq,(^y;y\in P),1\big)$ be a finite poset with pseudocomplemented sections and put $x\rightarrow y:=\big(\min U(x,y)\big)^y$ for all $x,y\in P$. Then $\mathbb I(\mathbf P):=(P,\rightarrow,1)$ is a finite {\rm I}-algebra.
\end{theorem}

\begin{proof}
Let $a,b\in P$. According to (i) of Proposition~\ref{prop1}, $a\leq b$ if and only if $a\rightarrow b=1$, and according to (ii) of Lemma~\ref{lem1}, $b\leq a$ implies $a\rightarrow b=a^b$. Now (I1) follows since $\leq$ is reflexive and $1$ is the top element of $(P,\leq)$, (I2) and (I3) follow by antisymmetry and transitivity of $\leq$, respectively. Let $x,y,z,u\in P$. If $y\leq z\leq x$ and $z\leq x^y$ then $z\in L(x,x^y)\cap[y,1]=\{y\}$, i.e.\ $z=y$ which shows that (I4) holds. Now for $x,u\in[y,1]$ the following are equivalent:
\begin{align*}
y\rightarrow z=z\rightarrow x=z\rightarrow u=1 & \Rightarrow z=y, \\
                          z\in L(x,u)\cap[y,1] & \Rightarrow z=y, \\
                               L(x,u)\cap[y,1] & \subseteq\{y\}, \\
                               L(x,u)\cap[y,1] & =\{y\}.
\end{align*}
Since for $x,u\in[y,1]$, $L(x,u)\cap[y,1]=\{y\}$ implies $u\leq x^y$, we have (I5). Finally, (I6) follows from the definition of $\rightarrow$.
\end{proof}

However, also the converse of Theorem~\ref{th1} is true, see the following result.

\begin{theorem}\label{th2}
Let $\mathbf A=(A,\rightarrow,1)$ be a finite {\rm I}-algebra and define
\begin{align*}
x\leq y & :\Leftrightarrow x\rightarrow y=1, \\
    x^y & :=x\rightarrow y\text{ whenever }y\leq x
\end{align*}
{\rm(}$x,y\in A${\rm)}. Then $\mathbb P(\mathbf A):=\big(A,\leq,(^y;y\in A),1\big)$ is a finite poset with pseudocomplemented sections.
\end{theorem}

\begin{proof}
Because of (I1) -- (I3), $(A,\leq,1)$ is a finite poset with top element $1$, because of (I4), $L(x,x^y)\cap[y,1]\subseteq\{y\}$ for all $x,y\in I$ with $y\leq x$ and hence $L(x,x^y)\cap[y,1]=\{y\}$ for all $x,y\in I$ with $y\leq x$, and because of (I5), $y\in A$, $x,u\in[y,1]$ and $L(x,u)\cap[y,1]=\{y\}$ imply $u\leq x^y$. Hence for all $y\in A$, $([y,1],\leq,{}^y,y)$ is a pseudocomplemented poset.
\end{proof}

\begin{remark}
In the above proof, condition {\rm(I6)} of Definition~\ref{def1} is not needed. We need this condition in order to prove that the above described correspondence is one-to-one.
\end{remark}

Now we show that the assignments from Theorems~\ref{th1} and \ref{th2} are mutually inverse.

\begin{theorem}
The correspondence described in Theorems~\ref{th1} and \ref{th2} is one-to-one.
\end{theorem}

\begin{proof}
Let $\mathbf P=\big(P,\leq,(^y;y\in P),1\big)$ be a finite poset with pseudocomplemented sections, put
\begin{align*}
                   \mathbb I(\mathbf P) & =(P,\rightarrow,1), \\
\mathbb P\big(\mathbb I(\mathbf P)\big) & =\big(P,\leq',(_y;y\in P),1\big)
\end{align*}
and let $a,b\in P$. Then because of the definition of $\leq'$ and (i) of Proposition~\ref{prop1} the following are equivalent:
\begin{align*}
             a & \leq'b, \\
a\rightarrow b & =1, \\
             a & \leq b.
\end{align*}
If $b\leq a$ then because of the definition of $a_b$ and (ii) of Lemma~\ref{lem1} we have $a_b=a\rightarrow b=a^b$ This shows $\mathbb P\big(\mathbb I(\mathbf P)\big)=\mathbf P$. Now let $\mathbf A=(A,\rightarrow,1)$ be a finite {\rm I}-algebra, put
\begin{align*}
                   \mathbb P(\mathbf A) & =\big(A,\leq,(^y;y\in I),1\big), \\
\mathbb I\big(\mathbb P(\mathbf A)\big) & =(A,\Rightarrow,1)
\end{align*}
and let $a,b\in A$. Then
\[
a\Rightarrow b=\big(\Min U(a,b)\big)^b=\{x^b\mid a,b\leq x,\text{ and }a,b\leq y\leq x\text{ implies }y=x\}=a\rightarrow b
\]
because of the definition of $\leq$ and (I6). This shows $\mathbb I\big(\mathbb P(\mathbf A)\big)=\mathbf A$.
\end{proof}

In every finite poset $\big(P,\leq,(^y;y\in P),0,1\big)$ with $0$ and pseudocomplemented sections one can define $\neg x:=x\rightarrow0$ for all $x\in P$. Observe that $\neg x=\max\{y\in P\mid L(x,y)=0\}$ for all $x\in P$ and hence $\neg x=x^0$ for all $x\in P$. Due to the fact that $\neg x$ is the pseudocomplementation as defined usually (see e.g.\ \cite{Ba} or \cite V), it satisfies the known properties as follows:
\begin{enumerate}[(P1)]
\item $\neg0=1$ and $\neg1=0$,
\item $x\leq y$ implies $\neg y\leq\neg x$,
\item $x\leq\neg\neg x$,
\item $\neg\neg\neg x=\neg x$.
\end{enumerate}

\begin{remark}
Condition {\rm(P2)} expresses the fact that our negation and implication satisfy the contraposition law, i.e.
\[
\text{if }x\rightarrow y=1\text{ then also }\neg y\rightarrow\neg x=1.
\]
\end{remark}

At the end of this section we show that every bounded pseudocomplemented poset contains a subposet where the unary negation $'$ is a complementation. This is in fact analogous to the Glivenko Theorem (see e.g.\ \cite{Bi}) for pseudocomplemented lattices.

\begin{proposition}
Let $\mathbf P=(P,\leq,{}',0,1)$ be a bounded pseudocomplemented poset. Then $(P',\leq,{}',0,1)$ with $P':=\{x'\mid x\in P\}$ is a complemented poset.
\end{proposition}

\begin{proof}
Clearly, $P'=\{x\in P\mid x''=x\}$. Let $a,b\in P'$. Then $a'\in P'$. Moreover, $L(a,a')=0$. If $b\in U(a,a')$ then $b'\in L(a',a'')=L(a,a')=0$ and hence $b=b''=0'=1$. This shows $U(a,a')=1$, i.e.\ $a'$ is a complement of $a$.
\end{proof}

\begin{example}
If $(P,\leq,{}',0,1)$ is the bounded pseudocomplemented poset of Example~\ref{ex1} then the complemented poset $(P',\leq,{}',0,1)$ is depicted in Figure~4:

\vspace*{-2mm}

\begin{center}
\setlength{\unitlength}{7mm}
\begin{picture}(6,6)
\put(3,1){\circle*{.3}}
\put(1,3){\circle*{.3}}
\put(5,3){\circle*{.3}}
\put(3,5){\circle*{.3}}
\put(1,3){\line(1,-1)2}
\put(1,3){\line(1,1)2}
\put(5,3){\line(-1,-1)2}
\put(5,3){\line(-1,1)2}
\put(2.85,.25){$0$}
\put(.3,2.85){$b$}
\put(5.4,2.85){$c$}
\put(2.85,5.4){$1$}
\put(2.2,-.75){{\rm Fig.~4}}
\end{picture}
\end{center}

\vspace*{4mm}

\end{example}

\section{Adjointness of implication with unsharp conjunction}

It is known that every relatively pseudocomplemented lattice is residuated, in fact it is a ``prototype'' of a residuated lattice where the operation multiplication is considered as the lattice meet and the relative pseudocomplement as a residuum. As mentioned in the introduction, we define sectionally pseudocomplemented lattices in the sake to extend the concept of relative pseudocomplementation to non-distributive lattices. The question concerning residuation in sectionally pseudocomplemented lattices was answered by the authors and J.~K\"uhr (\cite{CKL}) as follows.

A lattice $\mathbf L=(L,\vee,\wedge,\odot,\rightarrow,1)$ with top element $1$ and with two binary operations $\odot$ and $\rightarrow$ is called {\em relatively residuated} if
\begin{enumerate}[{\rm(i)}]
\item $(L,\odot,1)$ is a commutative groupoid with $1$,
\item $x\leq y$ implies $x\odot z\leq y\odot z$,
\item $(x\vee z)\odot(y\vee z)\leq z$ if and only if $x\vee z\leq y\rightarrow z$.
\end{enumerate}
It is worth noticing that the class of relatively residuated lattices forms a variety, see \cite{CKL}. Namely, under condition (i), conditions (ii) and (iii) are equivalent to the identities
\begin{enumerate}
\item[(iv)] $x\odot z\leq(x\vee y)\odot z$,
\item[(v)] $z\vee y\leq x\rightarrow\Big(\big((x\vee y)\odot(z\vee y)\big)\vee y\Big)$,
\item[(vi)] $(x\rightarrow y)\odot(x\vee y)\leq y$.
\end{enumerate}
Unfortunately, we cannot adopt this definition for posets $(P,\leq)$ because we cannot use the lattice operations and, moreover, our implication is not an operation but an operator, i.e.\ its result need not be a singleton. However, we can proceed as follows. Having in mind that $\rightarrow$ is unsharp, we can introduce an unsharp connective conjunction as follows:
\[
x\odot y:=\Max L(x,y)
\]
and for non-singleton subsets $A,B$ of $P$ we define $A\odot B:=\Max L(A,B)$. One can mention that this conjunction reaches the maximal possible values for given entries $x$ and $y$. Moreover, the operator $\odot$ is idempotent since for every $x\in P$ we have
\[
x\odot x=\Max L(x,x)=\Max L(x)=x.
\]

Now we can define the following concept.

\begin{definition}\label{def2}
A poset $\mathbf P=(P,\leq,\odot,\rightarrow,1)$ with top element $1$ and two operators $\odot$ and $\rightarrow$, both mappings from $P^2$ to $2^P$, such that
\begin{enumerate}[{\rm(i)}]
\item $\odot$ is commutative and associative and $x\odot1\approx x$,
\item if $x\leq y$ and $z\in P$ then there exists some $t\in y\odot z$ with $x\odot z\leq t$,
\item $z\in x\odot y$ if and only if $(z\leq x,y$ and $x\leq y\rightarrow z)$
\end{enumerate}
will be called {\em unsharply residuated}. Condition {\rm(iii)} will be called {\em unsharp adjointness}. We call an unsharply residuated poset $\mathbf P$ {\em divisible} if for all $x,y\in P$ with $x\geq y$ we have that $x\rightarrow y$ is a singleton and $\big(x\odot(x\rightarrow y)\big)\cap[y,1]=\{y\}$.
\end{definition}

We are going to show that finite posets with pseudocomplemented sections are unsharply residuated and divisible.

\begin{theorem}\label{th3}
Let $\big(P,\leq,(^y;y\in P),1\big)$ be a finite poset with pseudocomplemented sections and for $x,y\in P$ define
\begin{align*}
      x\odot y & :=\Max L(x,y), \\ 
x\rightarrow y & :=\big(\Min U(x,y)\big)^y.
\end{align*}
Then $(P,\leq,\odot,\rightarrow,1)$ is unsharply residuated and divisible.
\end{theorem}

\begin{proof}
Let $a,b,c\in P$. Then
\[
a\odot1=\Max L(a,1)=\Max L(a)=a
\]
and, clearly, $\odot$ is commutative. Moreover,
\begin{align*}
(a\odot b)\odot c & =\Max L\big(\Max L(a,b),c\big)=\Max\Big(L\big(\Max L(a,b)\big)\cap L(c)\Big)= \\
& =\Max\big(L(a,b)\cap L(c)\big)=\Max L(a,b,c)=\Max\big(L(a)\cap L(b,c)\big)= \\
& =\Max\Big(L(a)\cap L\big(\Max L(b,c)\big)\Big)=\Max L\big(a,\Max L(b,c)\big)=a\odot(b\odot c).
\end{align*}
Thus $\odot$ satisfies (i) of Definition~\ref{def2}. If $a\leq b$ then
\[
a\circ c=\Max L(a,c)\subseteq L(a,c)\subseteq L(b,c)
\]
and hence there exists some $d\in\Max L(b,c)$ with $a\odot c\leq d$. This shows (ii) of Definition~\ref{def2}. Now unsharp adjointness remains to be proved. Because of Lemma~\ref{lem1} (ii) the following are equivalent:
\begin{align*}
& c\in a\odot b, \\
& c\in\Max L(a,b), \\
& L(a,b)\cap[c,1]=\{c\}, \\
& c\leq a,b\text{ and }a\leq b^c, \\
& c\leq a,b\text{ and }a\leq b\rightarrow c.
\end{align*}
Now assume $a\geq b$. Then $a\rightarrow b=a^b$ and
\[
\big(a\odot(a\rightarrow b)\big)\cap[b,1]=\big(\Max L(a,a^b)\big)\cap[b,1]\subseteq L(a,a^b)\cap[b,1]=\{b\}.
\]
On the other hand, $b\in L(a,a^b)$ and if $b\leq c\in L(a,a^b)$ then $c\in L(a,a^b)\cap[b,1]=\{b\}$, i.e.\ $c=b$. This shows that $b\in\Max L(a,a^b)$ and hence $b\in\big(\Max L(a,a^b)\big)\cap[b,1]$, thus
\[
\big(\Max L(a,a^b)\big)\cap[b,1]=\{b\}.
\]
proving divisibility of $(P,\leq,\odot,\rightarrow,1)$.
\end{proof}

The divisibility has an essential influence on the logic for which the considered unsharply residuated poset is an algebraic semantics. Namely, if we know the truth values of $x$ and $x\rightarrow y$ and we know that $y\leq x$ then the truth value of $y$ is exactly the conjunction of $x$ and $x\rightarrow y$, which is just the derivation rule {\em Modus Ponens}.

If an unsharply residuated poset is a lattice then clearly we have
\[
x\odot y=\Max L(x,y)=\Max L(x\wedge y)=x\wedge y
\]
and the fact that $z\leq x,y$ can be expressed by $x\vee z=x$ and $y\vee z=y$. Then unsharp adjointness can be formulated as follows:
\[
(x\vee z)\odot(y\vee z)=z\text{ if and only if }x\vee z\leq(y\vee z)\rightarrow z.
\]
However, by (ii) of Proposition~\ref{prop1} we know that
\[
(y\vee z)\rightarrow z=y\rightarrow z,
\]
and
\[
(x\vee z)\odot(y\vee z)\geq z
\]
automatically holds. Hence the left-hand side of (iii) is equivalent to $(x\vee z)\odot(y\vee z)\leq z$. Altogether, we obtain
\[
(x\vee z)\odot(y\vee z)\leq z\text{ if and only if }x\vee z\leq y\rightarrow z
\]
which is just relative adjointness as defined in \cite{CKL} and mentioned above. This means that Definition~\ref{def2} is compatible with the corresponding definition for lattices.

\begin{example}
Let us consider the poset from Example~\ref{ex1}. The table for $\odot$ looks as follows:
\[
\begin{array}{c|ccccccc}
\odot & 0 & a & b & c & d & e & 1 \\
\hline
  0   & 0 & 0 & 0 & 0 &    0    &    0    & 0 \\
  a   & 0 & a & a & 0 &    a    &    a    & a \\
	b   & 0 & a & b & 0 &    b    &    b    & b \\
	c   & 0 & 0 & 0 & c &    c    &    c    & c \\
	d   & 0 & a & b & c &    d    & \{b,c\} & d \\
	e   & 0 & a & b & c & \{b,c\} &    e    & e \\
	1   & 0 & a & b & c &    d    &    e    & 1.
\end{array}
\]
We can see that $d\odot e=\{b,c\}$ is not a singleton, and $b\in d\odot e$ implies $b\leq d,e$ and $d\leq d=e\rightarrow b$; also, conversely, $c\leq e,d$ and $e\leq e=d\rightarrow c$ imply $c\in\{b,c\}=e\odot d$.
\end{example}

\section{Conclusion}

We constructed a binary operator on a finite poset with pseudocomplemented sections which can serve as an unsharp implication. It satisfies important properties required for implication in various sorts of propositional logics. Moreover, a negation derived by means of this implication satisfies the properties of implication in intuitionistic logic, thus our poset with this unsharp implication can be recognized as an algebraic semantics of a general case of intuitionistic logic. Moreover, an unsharp conjunction is introduced having similar properties as those satisfied by the connective conjunction in propositional calculus. This unsharp conjunction together with the mentioned unsharp implication forms an adjoint pair. Hence, the logic based on such a poset can be considered as a fairly general kind of substructural logic.

{\bf Declaration of competing interest}

The authors declare that they have no known competing financial interests or personal relationships that could have appeared to influence the work reported in this paper.

Authors' addresses:

Ivan Chajda \\
Palack\'y University Olomouc \\
Faculty of Science \\
Department of Algebra and Geometry \\
17.\ listopadu 12 \\
771 46 Olomouc \\
Czech Republic \\
ivan.chajda@upol.cz

Helmut L\"anger \\
TU Wien \\
Faculty of Mathematics and Geoinformation \\
Institute of Discrete Mathematics and Geometry \\
Wiedner Hauptstra\ss e 8-10 \\
1040 Vienna \\
Austria, and \\
Palack\'y University Olomouc \\
Faculty of Science \\
Department of Algebra and Geometry \\
17.\ listopadu 12 \\
771 46 Olomouc \\
Czech Republic \\
helmut.laenger@tuwien.ac.at
\end{document}